\documentclass[11pt]{article} 
\usepackage{amsmath, amsthm, amssymb,enumitem,graphicx, mathtools,bbm}
\setenumerate{label=(\roman*),itemsep=0pt}
\usepackage[margin=1.25in]{geometry}
\usepackage{xcolor}

\def\xx#1{{\color{red}#1}}
\def\xx#1{{#1}}

\title{The extremal function for bipartite linklessly embeddable graphs}

\begin{document}
\newtheorem{theorem}{Theorem}[section]
\newtheorem{conjecture}[theorem]{Conjecture}
\newtheorem{lemma}[theorem]{Lemma}
\newtheorem{claim}{Claim}[theorem]
\newtheorem{problem}[theorem]{Problem}

\theoremstyle{definition}
\newtheorem{definition}[theorem]{Definition}

\centerline{\Large \bf THE EXTREMAL FUNCTION FOR BIPARTITE}
\smallskip
\centerline{{\Large\bf LINKLESSLY EMBEDDABLE GRAPHS}%
\footnote{Partially supported by NSF under Grants No.~DMS-1202640 and~DMS-1700157. 2 July 2017, revised 17 December  2018.
}}

\bigskip
\bigskip

\centerline{{\bf Rose McCarty}%
%\footnote{Partially supported by NSF under Grant No.~DMS-1202640.}
}
\smallskip
\centerline{and}
\smallskip
\centerline{{\bf Robin Thomas}}
\bigskip
\centerline{School of Mathematics}
\centerline{Georgia Institute of Technology}
\centerline{Atlanta, Georgia  30332-0160, USA}
%end of authors
\bigskip

\begin{abstract}
\noindent
An embedding of a graph in $3$-space is {\em linkless} if for every two disjoint cycles there exists an 
embedded ball that contains one of the cycles and is disjoint from the other.
We prove that every bipartite linklessly embeddable (simple) graph on $n\ge5$ vertices 
has at most $3n-10$ edges, unless it is isomorphic to the complete bipartite graph $K_{3,n-3}$.
\end{abstract}

\section{Introduction}\label{intro}

All {\em graphs} in this paper are finite and simple. {\em Paths} and {\em cycles} have no ``repeated" vertices.
An embedding of a graph in $3$-space is {\em linkless} if for every two disjoint cycles there exists an 
embedded ball that contains one of the cycles and is disjoint from the other.
We prove the following theorem. 

\begin{theorem}
\label{thm:main}
Every bipartite linklessly embeddable  graph on $n\ge5$ vertices 
has at most $3n-10$ edges, unless it is isomorphic to the complete bipartite graph $K_{3,n-3}$.
\end{theorem}

\noindent
The question of whether linklessly embeddable bipartite graphs on $n\ge5$ vertices have at most $3n-9$ edges 
is stated as~\cite[Problem~2.3]{TaiCdV}, and Theorem~\ref{thm:main} is implied by~\cite[Conjecture~4.5]{KalNevNov}.
 
The following are equivalent conditions for a graph to be linklessly embeddable.
A graph $H$ is obtained from a graph $G$ by a {\em $Y\Delta$ transformation} if $H$ is obtained from $G$ by deleting a vertex~$v$ 
of degree 3 and joining every pair of non-adjacent neighbors of $v$ by an edge.
Conversely, $G$ is obtained from $H$ by means of   a {\em $\Delta Y$ transformation} if $G$ is obtained from $H$ by deleting the edges 
of a cycle of length 3 (``a triangle") and adding a vertex of degree 3 joined to the vertices of the triangle.
The {\em Petersen family} is the set of seven graphs obtained from the complete graph $K_6$ by means of  $Y\Delta$ and 
 $\Delta Y$ transformations. The Petersen graph is a member of the family, and hence the name. The Petersen family is depicted 
in Figure~\ref{fig:petfam}.
A graph is a {\em minor} of another if the first can be obtained from a subgraph of the second by contracting edges. 
An {\em $H$ minor} is a minor isomorphic to $H$.
We denote by $\mu(G)$ the graph invariant introduced by Colin de Verdi\`ere~\cite{ColMu}. We omit its definition,
because we do not need it.

\begin{figure}[htb]
 \centering
\includegraphics[scale = .75]{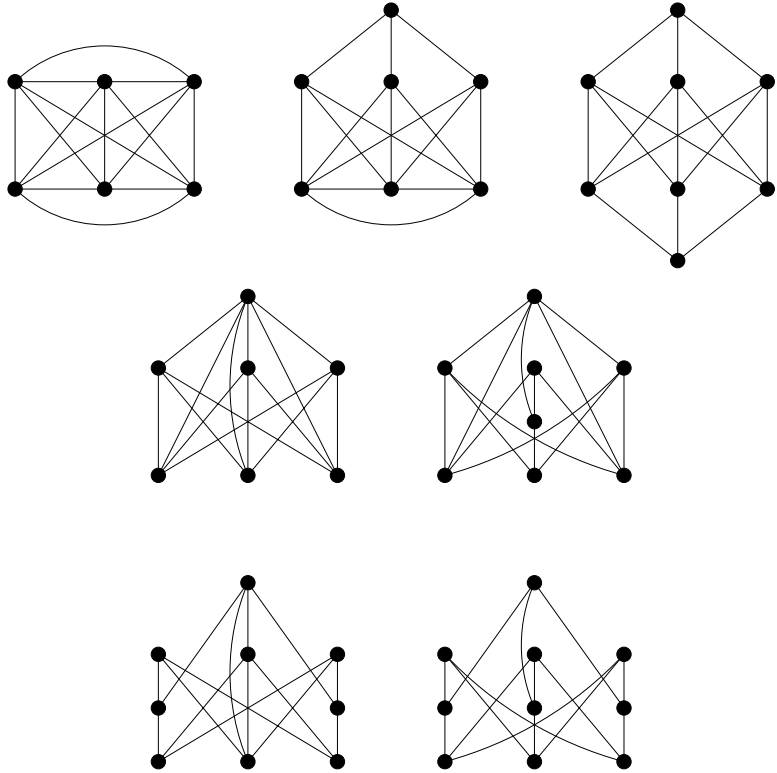}
 \caption{The Petersen family.}
\label{fig:petfam}
\end{figure}

\begin{theorem}
\label{thm:linkequiv}
For every graph $G$ the following conditions are equivalent:
\begin{itemize}
\item[\rm(i)]$G$ has an embedding in $3$-space such that every two disjoint cycles have even linking number.
\item[\rm(ii)]$G$ is linklessly embeddable.
\item[\rm(iii)]$G$ has an embedding in $3$-space such that every cycle bounds an open disk disjoint from the embedding of $G$.
\item[\rm(iv)]$G$ has no minor isomorphic to a member of the Petersen family.
\item[\rm(v)]$\mu(G)\le4$.
\end{itemize}
\end{theorem}

\noindent
Here (iii)$\Rightarrow$(ii) and (ii)$\Rightarrow$(i) are trivial, (i)$\Rightarrow$(iv) was shown by Sachs~\cite{Sachs1,Sachs2},
(iv)$\Rightarrow$(iii) was shown by Robertson, Seymour and the second author~\cite{RobSeyThoSachs}, 
(v)$\Rightarrow$(iv) was shown by Bacher and Colin de Verdi\`ere~\cite{BacCol}, and (iii)$\Rightarrow$(v)  was
shown by Lov\'asz and Schrijver~\cite{LovSchMu}.

Let us now put Theorem~\ref{thm:main} in perspective. For general graphs excluding only the Petersen graph as a minor, Hendrey and Wood \cite{HendreyWood} showed that the correct bound on the number of edges is $5n-9$, which is tight. For linklessly embeddable graphs that are not necessarily bipartite the correct bound 
on the number of edges is $4n-10$, which is tight for any graph obtained from a planar triangulation $G$ on at least three vertices by adding a new vertex  with neighborhood $V(G)$. This bound follows from the following more general result of Mader~\cite{MadHom}.

\begin{theorem}\label{thm:mader}
For every integer $p=2,3,\ldots,7$, every graph on $n\ge p-1$ vertices with no minor isomorphic to $K_p$ has
 at most $(p-2)n-{p-1\choose2}$ edges.
\end{theorem}

Theorem~\ref{thm:mader} is such a nice result that it raises the question of whether it
can be generalized to all values of $p$. But there is some depressing news:
 for large $p$ a graph must have at least
$\Omega(p\sqrt{\log p}n)$ edges in order to guarantee a $K_p$ minor,
because, as noted by several people (Kostochka~\cite{Kosminor,Koshad}, and
Fernandez de la Vega \cite{FerdlV} based on Bollob\'as, Catlin and Erd\"os \cite{BolCatErd}),
a random graph with no $K_p$ minor may have average degree of order
$p\sqrt{\log p}$. Kostochka~\cite{Kosminor,Koshad} and Thomason \cite{Thocontr}
 proved that this is indeed the
correct order of magnitude, and in a remarkable result, Thomason~\cite{Thominors}  was
able to determine the constant of proportionality. For small $s$, K\"{u}hn and Osthus \cite{KuhnOsthus} and Kostochka and Prince \cite{KosPrince} have shown that average degree of order $p$ suffices to guarantee a $K_{s,p}$ minor.

It may seem that an effort to generalize Theorem~\ref{thm:mader} to clique minors will be in vain,
but there are still the following possibilities. The random graph examples
provide only finitely many counterexamples for any given value of $p$. Of course,
more counterexamples can be obtained by taking disjoint unions or even
gluing counterexamples along small cutsets, but we know of no construction
of highly connected infinite families of counterexamples.
More specifically, Seymour and the second author conjecture the following.

\begin{conjecture}\label{conj:pdsrt} For every integer $p\ge2$ there exists a constant $N=N(p)$ such that
every $(p-2)$-connected graph on $n\ge N$ vertices with no minor isomorphic to $K_p$ has at most 
$(p-2)n-{p-1\choose2}$ edges.
\end{conjecture}

In a slightly different direction the first author conjectures~\cite{McC} the following.

\begin{conjecture}\label{conj:rm}
For every integer $p\ge3$, every graph $G$ on $n\ge p-1$ vertices with $\mu(G)\le p-2$ has
 at most $(p-2)n-{p-1\choose2}$ edges.
\end{conjecture}

\noindent
Whether Conjecture~\ref{conj:rm} holds is stated as~\cite[Problem~1]{TaiCdV}. 
 Conjecture~\ref{conj:rm} is implied by~\cite[Conjecture~1.5]{Nev}.

Let us  repeat that for not necessarily bipartite graphs the bound on the number of edges for linklessly 
embeddable graphs and graphs with no $K_6$ minors coincide. Not so for bipartite graphs. 
\xx{In an earlier version of this paper we conjectured the following.}

\begin{conjecture}\label{conj:kp}
For every integer $p=2,3,\ldots,8$, every bipartite  graph on $n\ge 2p-5$ vertices with no minor isomorphic to $K_p$ has
 at most $(p-2)n-(p-2)^2$ edges.
\end{conjecture}

\noindent
The bound in  Conjecture~\ref{conj:kp} is tight, because of the graphs 
$K_{p-2,n-p+2}$. For $p\le4$ Conjecture~\ref{conj:kp} is easy, and for $p=5$ it follows from Wagner's
characterization of graphs with no $K_5$ minor~\cite{Wag37}. Conjecture~\ref{conj:kp} certainly does not 
hold for all $p$, because a graph with $\Omega(p\sqrt{\log p}n)$ edges and no  $K_p$ minor
has a bipartite subgraph  with $\Omega(p\sqrt{\log p}n)$ edges and no  $K_p$ minor. Since the time of submission, Thomas and Yoo \cite{ThomasYoo} proved a theorem implying Conjecture \ref{conj:kp}. They proved

\xx{%
\begin{theorem}\label{thm:kp}
For every integer $p=2,3,\ldots,9$, every triangle-free graph on $n\ge 2p-5$ vertices with no minor isomorphic to $K_p$ has
 at most $(p-2)n-(p-2)^2$ edges.
\end{theorem}
}

% that the conjecture holds for triangle-free graphs for every integer $p=2,3,\ldots,9$.

Motivated by Theorem~\ref{thm:main} and the equivalence of
(ii) and (v) in Theorem~\ref{thm:linkequiv} we also conjecture the following.

\begin{conjecture}\label{conj:mu}
For every integer $p\ge3$, every bipartite graph $G$ on $n\ge 2p-3$ vertices with $\mu(G)\le p$ has
 at most $(p-1)n-(p-1)^2$ edges.
\end{conjecture}

\noindent
Let us remark that the bound in  Conjecture~\ref{conj:mu}, if true,  is tight,  because of the graphs 
$K_{p-1,n-p+1}$. For $p=3$ Conjecture~\ref{conj:mu} follows from the fact that graphs $G$ with 
$\mu(G)\le3$ are precisely planar graphs~\cite{ColMu}, and for $p=4$ it  follows from  Theorems~\ref{thm:main}
and~\ref{thm:linkequiv}.

For linklessly embeddable graphs, we conjecture that Theorem~\ref{thm:main} in fact holds for triangle-free graphs.

\begin{conjecture}
\label{conj:trfree}
Every triangle-free linklessly embeddable  graph on $n\ge5$ vertices 
has at most $3n-10$ edges, unless it is isomorphic to the complete bipartite graph $K_{3,n-3}$.
\end{conjecture}

\noindent
A possible approach to Conjecture~\ref{conj:trfree} is to prove the following conjecture:

\begin{conjecture}
\label{conj:trfull}
Every linklessly embeddable  graph on $n\ge7$ vertices with $t$ triangles
has at most $3n-9+t/3$ edges.
\end{conjecture}
\noindent Thomas and Yoo \cite{ThomasYoo} recently proved that Conjecture \ref{conj:trfull} holds for apex graphs, that is, graphs $G$ with a vertex $v$ so that $G-v$ is planar. One could speculate whether Conjecture~\ref{conj:mu} holds for triangle-free graphs,
but we do not have enough evidence to formally conjecture that.

The paper is organized as follows. 
In the next section we introduce definitions and notation.
In Section~\ref{sec:proof1} we state Theorem~\ref{thm:main2}, which implies Theorem~\ref{thm:main}
and prove half of it, proving some useful lemmas and disposing of vertices of degree 5.
In Section~\ref{sec:proof2} we complete the proof of Theorem~\ref{thm:main2} by disposing of vertices of degree 4.

\section{Notation and Definitions}
\label{sec:notation}
For positive integers $n_1, n_2, \ldots, n_k$ with $k \geq 2$, we let $K_{n_1, n_2, \ldots, n_k}$ denote the complete multipartite graph with $k$ independent sets of sizes $n_1, n_2, \ldots, n_k$. We let $K_{4,4}^-$ denote the graph obtained from
 $K_{4,4}$ by deleting an edge. We also let $K_6^{\Delta Y}$ denote the graph obtained from $K_6$ by performing a $\Delta Y$ transformation.

For a graph $G$ we write $V(G)$ for the vertex set of $G$ and $E(G)$ for the edge set of $G$. We write $\delta(G)$ for the minimum degree of $G$ and $\Delta(G)$ for the maximum degree of $G$. Suppose $v$ is a vertex of $G$ and $S$ is a subset of $V(G)$. Then we write $G[S]$ for the induced subgraph of $G$ with vertex set $S$ and $G-S$ for the induced subgraph of $G$ with vertex set $V(G)-S$. We write $G-v$ for $G-\{v\}$. We write $d_G(v)$, or $d(v)$ if the graph is understood from context, for the degree of $v$ in $G$. We write $N_G(S)$ for the set of all vertices in $V(G)-S$ that are adjacent to some vertex in $S$. We write $N(S)$ if the graph is understood from context, and we write $N(v)$ for $N(\{v\})$. We use $N[v]$ to denote $N(v) \cup \{v\}$. 

If $G$ is a graph with $S$ and $T$ disjoint subsets of $V(G)$, we say an edge $uv \in E(G)$ is \textit{between} $S$ and $T$ if $S\cap \{u,v\} \neq \emptyset$ and $T\cap \{u,v\} \neq \emptyset$. If $S$ consists of a single vertex $v$, we may talk about the edges between $v$ and $T$. Given a graph $G$, we say that $\{X_0, X_1\}$ is a \textit{bipartition} of $G$ if $\{X_0, X_1\}$ is a partition of $V(G)$ so that all edges of $G$ are between $X_0$ and $X_1$. 

We define a \textit{separation} of a graph $G$ to be a pair of sets $(A,B)$ with union $V(G)$ such that $G$ has no edge between $A-B$ and $B-A$. The \textit{order} of a separation $(A,B)$ is $|A \cap B|$. We also say that a separation of order $k$ is a $k$-separation. A separation $(A,B)$ is \textit{non-trivial} if both $A-B$ and $B-A$ are non-empty. We say that a separation $(A,B)$ is \textit{minimal} if there does not exist a non-trivial separation $(A',B')$ of $G$ with $A'\cap B' \subsetneq A \cap B$.

It is convenient for us to give the following related definition. We say a \textit{super-separation} of a graph $G$ is a pair of graphs $(G_0, G_1)$ such that $V(G)\subseteq V(G_0)\cup V(G_1)$, and $E(G) \subseteq E(G_0)\cup E(G_1)$, and both $G_0$ and $G_1$ are isomorphic to minors of $G$. We say a super-separation $(G_0, G_1)$ of $G$ is \textit{non-trivial} if both $G_0$ and $G_1$ are isomorphic to proper minors of $G$. (That is, neither $G_0$ nor $G_1$ is isomorphic to $G$.) We say that the \textit{order} of a super-separation $(G_0, G_1)$ of $G$ is $|V(G_0)|+ |V(G_1)|-|V(G)|$. Finally we say a super-separation $(G_0, G_1)$ is \textit{bipartite} if both $G_0$ and $G_1$ are bipartite. 

Note that if $(A,B)$ is a (non-trivial) separation of $G$ of order $k$, then $(G[A], G[B])$ is a (non-trivial) super-separation of $G$ of order $k$. Furthermore, if $G$ is bipartite then the super-separation $(G[A],G[B])$ is bipartite. In this paper, each super-separation we use will be constructed from a non-trivial separation $(A,B)$ as follows. We will construct a graph $G_A$ formed from $G[A]$ by possibly adding some edges with both ends in $A \cap B$, and possibly a new vertex $a\notin V(G)$ with neighbors in $A \cap B$. We will show that $G_A$ is isomorphic to a proper minor of $G$ by contracting some edges with at least one end in $B$. A graph $G_B$ will be formed similarly from $G[B]$, so that $(G_A, G_B)$ is a non-trivial super-separation. 

Finally, if $G$ is a bipartite graph with bipartition $\{X_0, X_1\}$ and $S \subseteq V(G)$, then we will write $\overline{G[S]}$ for the \textit{bipartite complement} of $G[S]$. That is, $\overline{G[S]}$ is the graph on vertex set $S$ where $uv$ is an edge of $\overline{G[S]}$ if and only if exactly one of $u$ and $v$ is in $X_0$ and $uv \notin E(G)$.

\section{Proof of Main Theorem: Vertices of Degree 5}
\label{sec:proof1}
By Theorem \ref{thm:linkequiv}, the following theorem implies Theorem \ref{thm:main}.
\begin{theorem}
\label{thm:main2}
Every bipartite graph on $n\geq 5$ vertices with no $K_6$, $K_{1,3,3}$, $K_{4,4}^-$, or $K_6^{\Delta Y}$ minor has at most $3n-10$ edges, unless it is isomorphic to the complete bipartite graph $K_{3,n-3}$.
\end{theorem}

The rest of the paper is dedicated to proving Theorem \ref{thm:main2}.
%
%\begin{proof}
Going for a contradiction, suppose that Theorem \ref{thm:main2} is false. Let $G$ be a counterexample with $|V(G)|+|E(G)|$ minimum. Write $n \coloneqq |V(G)|$, and let $\{X_0, X_1\}$ be a bipartition of $G$.

We begin by giving a brief outline of our proof strategy. First we will show an easy lemma, and that $4 \leq \delta(G) \leq 5$. Then we show that $G$ cannot have certain separations and super-separations of small order. It follows that $G$ has no subgraph isomorphic to $K_{3,3}$: otherwise it either has a $K_{1,3,3}$ minor or a separation of small order. Next we show that if $v$ is a vertex of degree 4 or 5 and $x$ and $y$ are neighbors of $v$, then $x$ and $y$ have several common neighbors other than $v$. Then it is fairly easy to show that $G$ has no vertex $v$ of degree 5: for every pair of distinct neighbors $x$ and $y$ of $v$, let $v_{x,y}$ be a vertex other than $v$ that is adjacent to both $x$ and $y$. If all ten $v_{x,y}$ are distinct, then $G$ has a $K_6$ minor. Otherwise we find a $K_{3,3}$ subgraph or another forbidden minor. In Section \ref{sec:proof2} we will deal with the case that $\delta(G)=4$. 

We begin with two easy lemmas:

\begin{lemma}
\label{noVertices}
$n \geq 7$ 
\end{lemma}
\begin{proof}
Otherwise, $n \in \{5,6\}$. Then $\lceil n/2 \rceil = 3$ and $\lfloor n/2 \rfloor = n-3$. If $G$ is a subgraph of $K_{3,n-3}$, then since by assumption $G$ is not isomorphic to $K_{3,n-3}$, we have $|E(G)|\leq |E(K_{3,n-3})|-1=3n-10$, a contradiction. So $G$ is not a subgraph of $K_{3, n-3}$. Then
\begin{align*}3n-9&\leq |E(G)| 
\leq |X_0||X_1| 
\leq (\lceil n/2 \rceil+1)(\lfloor n/2 \rfloor-1)
=4(n-4).\end{align*}
This gives us $n \geq 7$, a contradiction.
\end{proof}

\begin{lemma}
\label{minDegree}
$4\leq \delta(G) \leq 5$
\end{lemma}
\begin{proof}
Since $G$ was chosen to be a counterexample with $|V(G)|+|E(G)|$ minimum, $|E(G)| \leq 3n-8$. So $\delta(G)\leq 5$ by the handshaking lemma.

Now, let $v$ be a vertex of minimum degree. Since $n \geq 6$ by Lemma \ref{noVertices}, either $G-v$ is isomorphic to $K_{3,n-4}$ and $|E(G-v)|=3(n-1)-9$, or $|E(G-v)| \leq 3(n-1)-10$. If $d(v)\leq 2$, then
\begin{align*}|E(G)|&=|E(G-v)|+d(v)
\leq 3(n-1)-9+2
=3n-10,\end{align*}
a contradiction. Now suppose that $d(v)=3$. If $G-v$ is not isomorphic to $K_{3,n-4}$, then similarly $|E(G)|\leq 3n-10$, and we are done. 

So $G-v$ is isomorphic to $K_{3,n-4}$. Without loss of generality suppose that $v \in X_0$. If $N(v) = X_1$, then $G$ is isomorphic to $K_{3,n-3}$, a contradiction. So there exists a vertex $u \in X_1-N(v)$. Then $|X_{0}-\{v\}|=3$, and $G[X_0 \cup \{u\}\cup N(v)]$ is isomorphic to $K_{4,4}^-$, a contradiction. So $\delta(G) = d(v)\geq 4$, completing the proof of the lemma.
\end{proof}

Next we prove two lemmas on separations and super-separations of $G$. Observe that since $\delta(G) \geq 4$ by Lemma \ref{minDegree}, if $(A,B)$ is a non-trivial separation of $G$, then $(G[A],G[B])$ is a non-trivial bipartite super separation of $G$ such that $|V(G[A])|, |V(G[B])| \geq 5$. We will frequently apply the following lemma to such a case.

\begin{lemma}
\label{separation}
Let $(G_0, G_1)$ be a non-trivial bipartite super-separation of $G$ of order $k$ such that $|V(G_0)|, |V(G_1)| \geq 5$ and neither $G_0$ nor $G_1$ is isomorphic to $K_{3,t}$ for any $t$. Then $3k \geq |E(G_0)|+|E(G_1)| -|E(G)|+11$.
\end{lemma}
\begin{proof} For convenience, write $e = |E(G_0)|+|E(G_1)| -|E(G)|$. By the conditions of the lemma and since $G$ is a counterexample with $|V(G)|+|E(G)|$ minimum, \begin{align*}
3n-9&\leq |E(G)|\\
&= |E(G_0)|+|E(G_1)|-e \\
&\leq 3(|V(G_0)|+|V(G_1)|)-20-e
\\&=3(n+k)-20-e.
\end{align*}
So $3k \geq e+11$, as desired.
\end{proof}

Next we show that $G$ does not have certain separations of small order.
\begin{lemma}
\label{connectivity}
Let $(A, B)$ be a non-trivial separation of $G$ such that for each $i \in \{0,1\}$, $|A \cap B\cap X_i| \leq 3$. Then $|A \cap B|=6$ and $\Delta(G[A \cap B]) \leq 1$.
\end{lemma}
\begin{proof}
Suppose otherwise for some separation $(A,B)$. Note that any non-trivial separation $(A',B')$ of $G$ with $A'\cap B' \subsetneq A \cap B$ also violates the lemma. Thus we may assume that $(A, B)$ is minimal.

First we show that both $A$ and $B$ have at least four vertices in each side of the bipartition of $G$. Let $v \in A-B$, and without loss of generality assume that $v \in X_0$. Then $|X_1 \cap A| \geq |N(v)| \geq 4$ since $\delta(G)\geq 4$ by Lemma \ref{minDegree}. Also, since $|A \cap B \cap X_1| \leq 3$, there exists a vertex $u \in N(v)-(A \cap B)$. Then similarly $|X_0\cap A|\geq |N(u)|\geq 4$. The same argument shows that $B$ has at least four vertices in each side of the bipartition of $G$.

Now for convenience write $S \coloneqq A \cap B$. So $|S| \leq 6$. Let $z \in S$ so that $d_{\overline{G[S]}}(z)$ is maximum, where $\overline{G[S]}$ is the bipartite complement of $G[S]$. Let $G_A$ be the graph formed from $G[A]$ by adding edges between $z$ and every vertex in $N_{\overline{G[S]}}(z)$. We can see that $G_A$ is a minor of $G$ by contracting some component of $G[B-A]$ to $z$ and by the minimality of $(A,B)$. Furthermore $G_A$ is bipartite, has fewer vertices than $G$, and has at least four vertices in each side of the bipartition of $G$. So $G_A$ is not isomorphic to $K_{3,t}$ for any $t$. Define $G_B$ analogously, by adding edges between $z$ and every vertex in $N_{\overline{G[S]}}(z)$ to $G[B]$.

We have shown that $(G_A, G_B)$ is non-trivial bipartite super-separation of $G$ so that $G_A$ and $G_B$ both have at least five vertices, and neither $G_A$ nor $G_B$ is isomorphic to $K_{3,t}$ for any $t$. Furthermore, the order of $(G_A, G_B)$ is $|S|$ and
\begin{align*}|E(G_A)|+|E(G_B)|-|E(G)|&=|E(G[A])|+|E(G[B])|+2d_{\overline{G[S]}}(z) -|E(G)|
\\&= |E(G[S])|+2d_{\overline{G[S]}}(z).
\end{align*}

So by Lemma \ref{separation} applied to the super-separation $(G_A, G_B)$, we have $3|S|\geq |E(G[S])|+2d_{\overline{G[S]}}(z) +11$. Thus $|S| \geq 4$. We proceed by cases.

\begin{flushleft}
\textit{Case:} $|S|=4$
\end{flushleft}

Then $|E(G[S])|+2d_{\overline{G[S]}}(z)\leq 1$. So $d_{\overline{G[S]}}(z)=0$. Thus $G[S]$ is a complete bipartite graph on four vertices with at least one vertex on each side of its bipartition. So $|E(G[S])|\geq 3$, a contradiction.

\begin{flushleft}
\textit{Case:} $|S|=5$
\end{flushleft}
Since $|X_0 \cap S|, |X_1 \cap S| \leq 3$ and by symmetry between $X_0$ and $X_1$, we may assume that $|X_0\cap S|=3$ and $|X_1\cap S|=2$. Let $z_1$ and $z_2$ be the vertices in $X_1 \cap S$. By the definition of $z$,\begin{align*}
    4&\geq |E(G[S])|+2d_{\overline{G[S]}}(z)
    \geq |E(G[S])|+d_{\overline{G[S]}}(z_1)+d_{\overline{G[S]}}(z_2)= 6,
\end{align*}a contradiction.

\begin{flushleft}
\textit{Case:} $|S|=6$ and $\Delta(G[S])\geq 2$
\end{flushleft}

Then $|X_0\cap S|=|X_1\cap S|=3$. Let $z_1$ and $z_2$ be the other vertices on the same side of the bipartition of $G[S]$ as the vertex of maximum degree. Then \begin{align*}
    7\geq |E(G[S])|+2d_{\overline{G[S]}}(z)
    \geq |E(G[S])|+d_{\overline{G[S]}}(z_1)+d_{\overline{G[S]}}(z_2)
= \Delta(G[S])+6
    \geq 8,
\end{align*}a contradiction.
\end{proof}

Next we observe that $G$ has no $K_{3,3}$ subgraph, and then we show that common neighbors of a vertex of degree 4 or 5 in fact share several common neighbors.

\begin{lemma}
\label{noK33}
$G$ does not have a subgraph isomorphic to $K_{3,3}$.
\end{lemma}
\begin{proof}
Suppose $H$ is a subgraph of $G$ isomorphic to $K_{3,3}$. Since $n \geq 7$ by Lemma \ref{noVertices}, the graph $G-V(H)$ is non-empty. Let $C$ be the vertex set of some component of $G-V(H)$. Then since $(C \cup N(C)) \cap (V(G)-C) = N(C) \subseteq V(H)$, by Lemma \ref{connectivity}, the separation $(C \cup N(C), V(G)-C)$ is trivial. Then $C \cup N(C)=V(G)$, and so $N(C) = V(H)$. So the graph obtained by contacting $C$ to a single vertex is isomorphic to $K_{1,3,3}$, a contradiction.
\end{proof}

\begin{lemma}
\label{inCommon}
Let $v \in V(G)$ be a vertex of degree 4 or 5. Let $x$ and $y$ be distinct vertices in $N(v)$. Then $x$ and $y$ share at least $7-d(v)$ common neighbors other than $v$. 
\end{lemma}
\begin{proof}
Suppose otherwise, and write $c \coloneqq |N(x) \cap N(y)|-1$. That is, $c$ is the number of common neighbors of $x$ and $y$ other than $v$. So $c \leq 6-d(v)$. Without loss of generality suppose that $v \in X_0$. Let $G'$ be the graph formed from $G$ by deleting $y$ and $v$, and adding edges between $x$ and all vertices in $N(y)-N(x)$. We can see that $G'$ is a minor of $G$ by contracting $y$ and $v$ to $x$. Furthermore, $G'$ is bipartite and since $n \geq 7$ by Lemma \ref{noVertices}, the graph $G'$ has at least five vertices. Let $\ell$ be 1 if $G'$ is isomorphic to $K_{3,t}$ for some $t$, and 0 otherwise. Then:
\begin{align*}
3n-9&\leq|E(G)|\\
&=|E(G-v)|+d(v) \\
&= |E(G')|+c+d(v)\\
&\leq 3(n-2)-10+\ell+c+d(v)\\
& =3n-10+\ell+(c-6+d(v)).
\end{align*}
It follows that $\ell= 1$ and $c = 6-d(v)$. Thus, $G'$ is isomorphic to $K_{3,t}$ for some $t$. If $d(v) = 4$, then $c=2$ and $G[N(v)\cup (N(x)\cap N(y))]$ is isomorphic to $K_{3,4}$. This is a contradiction since by Lemma \ref{noK33}, the graph $G$ has no $K_{3,3}$ subgraph. If $d(v)=5$, then $|X_{1}\cap V(G')|\geq d(v)-1=4$, and so $|X_0 \cap V(G')| = 3$. Then $G[(X_0 \cap V(G'))\cup (N(v)-\{x,y\})]$ is isomorphic to $K_{3,3}$, again a contradiction to Lemma \ref{noK33}.
\end{proof}

Now we are ready to show:

\begin{lemma}
\label{no5}
$G$ has no vertex of degree 5.
\end{lemma}
\begin{proof}
Suppose $v \in V(G)$ is a vertex of degree 5. Let $$W \coloneqq \{w \in V(G)-N[v]: |N(w) \cap N(v)|= 2\},$$ 
and let $$U_0 \coloneqq \{u \in V(G)-N[v]: |N(u) \cap N(v)|\geq 3\}.$$ We will show a contradiction by proving that $G[N[v]\cup W \cup U_0]$ has a minor isomorphic to $K_6$ or $K_6^{\Delta Y}$. If $U_0 = \emptyset$, then this is immediate since by Lemma 3.8 every pair of vertices in $N(v)$ share at least two common neighbors other than $v$. On the other hand, if $U_0$ has too many vertices with too many neighbors in $N(v)$, then we will find a $K_{3,3}$ subgraph, contradicting Lemma 3.7. The proof proceeds by carefully contracting certain vertices in $W\cup U_0$ to one of their neighbors in $N(v)$. Note that since $G$ is bipartite, if $G'$ is obtained from $G[N[v]\cup W \cup U_0]$ by performing such contractions and then deleting edges with both ends in $N(v)$, then $G'$ is a subgraph of $G$. We will sometimes need this fact to find a $K_{3,3}$ subgraph in $G$.

Now, let $G_0$ be the graph formed from $G[N[v]\cup W \cup U_0]$ by contracting, for every vertex $w \in W$, an arbitrary edge with one end $w$ and the other end in $N(w) \cap N(v)$. 
\xx{Please note that $G_0$ is not necessarily bipartite.}
By Lemma \ref{inCommon}, every pair of vertices in $N(v)$ are either adjacent in $G_0$ or share at least two common neighbors in $U_0$. 

First we show the following claim:
\begin{claim}
\label{claim:deg5}
There exist a set $U_1 \subseteq U_0$ and a graph $G_1$ so that:
\begin{enumerate}
\item The graph $G_1$ is formed from $G_0$ by contracting edges with one end  in $U_0-U_1$ and the other end  in $N(v)$.
\item Every pair of distinct vertices in $N(v)$ are either adjacent in $G_1$ or share a common neighbor in $U_1$.
\item Every vertex in $U_1$ has degree exactly 3 in $G_0$, and $\delta(G_1[N(v)]) \geq 1$. 
\end{enumerate}
\end{claim}
\begin{proof}
Observe that $U_0$ is non-empty since otherwise $G_0[N[v]]$ is isomorphic to $K_6$. Fix a vertex $z \in U_0$ with $d_{G_0}(z)$ maximum. First suppose $d_{G_0}(z)=5$. Then since $G$ has no $K_{3,3}$ subgraph by Lemma \ref{noK33}, every pair of vertices in $N(v)$ are adjacent in $G_0$. Then $G_0[N[v]]$ is isomorphic to $K_6$, a contradiction. So $d_{G_0}(z)\leq 4$.

Now observe that every vertex in $U_0$ other than $z$ has degree exactly 3 in $G_0$. This is clear if $d_{G_0}(z)=3$, and follows since $G$ has no $K_{3,3}$ subgraph if $d_{G_0}(z)= 4$.

Let $x \in N(v)-N_{G_0}(z)$. If $d_{G_0}(z)=3$, let $x'$ be the vertex other than $x$ in $N(v)-N_{G_0}(z)$. If $d_{G_0}(z)=4$, let $x'$ be any vertex in $N(v)$ other than $x$.

First suppose that $x$ and $x'$ are adjacent in $G_0$. Then let $G_1$ be the graph formed from $G_0$ by contracting $z$ to one of its neighbors in $G_0$, and let $U_1 \coloneqq U_0-\{z\}$. Then $G_1$ and $U_1$ satisfy the conditions of the claim.

So we may assume that $x$ and $x'$ are not adjacent in $G_0$. Then they have a common neighbor $z' \in U_0-\{z\}$. Let $G_1$ be the graph formed from $G_0$ by contracting $z$ to a vertex in $N_{G_0}(z)-N_{G_0}(z')$  and $z'$ to $x'$. Write $U_1 \coloneqq U_0-\{z,z'\}$. Then $G_1$ and $U_1$ satisfy the conditions of the claim.
\end{proof}

Fix $G_1$ and $U_1$ as in the claim. Choose a graph $G_2$ and a set $U_2 \subseteq U_1$ so that:
\begin{enumerate}
\item The graph $G_2$ is formed from $G_1$ by contracting edges with one end  in $U_1-U_2$ and the other end  in $N(v)$.
\item Every pair of distinct vertices in $N(v)$ are either adjacent in $G_2$ or share a common neighbor in $U_2$.
\item Subject to the above, $|U_2|$ is minimum.
\end{enumerate}

Such a choice is possible because $G_2 \coloneqq G_1$ and $U_2 \coloneqq U_1$ satisfy (i) and (ii). Observe that $G_2$ is a minor of $G$. We first show that for all $u \in U_2$, the graph $G_2[N_{G_2}(u)]$ has no edges. Since every vertex in $U_1$ has degree exactly 3 in $G_0$ by the claim, $u$ also has degree exactly 3 in $G_2$. Write $N_{G_2}(u) = \{x, x', x''\}$ and suppose $xx' \in E(G_2)$. Then let $G_2'$ be the graph formed from $G_2$ by contracting $u$ to $x''$, and let $U_2' \coloneqq U_2-\{u\}$. Then $G_2'$ and $U_2'$ contradict our choice of $G_2$ and $U_2$.

Then by the last paragraph and condition (ii), if $|U_2| \leq 1$, then $G_2[N[v]\cup U_2]$ is isomorphic to either $K_6$ or $K_6^{\Delta Y}$. So there exist distinct vertices $u, u' \in U_2$. Both $u$ and $u'$ have degree exactly 3 in $G_2$. We proceed by cases.

\begin{flushleft}
\textit{Case:} $|N_{G_2}(u) \cap N_{G_2}(u')|=3$
\end{flushleft}

Then $G[N_{G_2}(u)\cup \{v,u,u'\}]$ is isomorphic to $K_{3,3}$, a contradiction to Lemma \ref{noK33}. 

\begin{flushleft}
\textit{Case:} $|N_{G_2}(u) \cap N_{G_2}(u')|=2$
\end{flushleft}
Then let $x$ be the unique vertex in $N_{G_2}(u)-N_{G_2}(u')$. Let $G_2'$ be the graph formed from $G_2$ by contracting $u$ to $x$, and let $U_2' \coloneqq U_2-\{u\}$. Then $G_2'$ and $U_2'$ contradict part (iii) of our choice of $G_2$ and $U_2$.

\begin{flushleft}
\textit{Case:} $|N_{G_2}(u) \cap N_{G_2}(u')|=1$
\end{flushleft}
Let $x$ be the unique vertex in $N_{G_2}(u) \cap N_{G_2}(u')$. Then $x$ is adjacent to no vertices in $N(v)$ in the graph $G_2$. But this is a contradiction since $\delta(G_2[N(v)])\geq \delta(G_1[N(v)])\geq 1$ by part (iii) of Claim \ref{claim:deg5}. This is the final case and completes the proof of Lemma \ref{no5}.
\end{proof}

\section{Proof of Main Theorem: Vertices of Degree 4}
\label{sec:proof2}

Now that we have shown $G$ has no vertices of degree 5 and that $4 \leq \delta(G) \leq 5$ by Lemma \ref{minDegree}, the remainder of the proof deals with vertices of degree 4. First we will show that if $v$ is any vertex of degree 4, then $G$ has no vertex $u$ such that $|N(u) \cap N(v)|\geq 4$. We then use this fact to show that $G$ does not have additional kinds of separations of small order. Finally we fix a vertex $v$ of degree 4 and a certain set $U \subseteq V(G)-N[v]$ of three or fewer vertices that each have neighbors in $N(v)$. We show that $G-(N[v]\cup U)$ is connected and has a cut vertex $a$. We then use the fact that $G-(N[v]\cup U\cup \{a\})$ is disconnected to find a separation showing a contradiction to Lemma \ref{separation}.

\begin{lemma}
\label{no44}
Suppose $v \in V(G)$ is a vertex of degree 4. Then there does not exist a vertex $u \in V(G)-N[v]$ so that $|N(u)\cap N(v)|=4$.
\end{lemma}
\begin{proof}
Suppose otherwise. Without loss of generality assume that $v \in X_0$. Write $N(v) = \{v_1, v_2, v_3, v_4\}$. For every $i,j \in \{1,2,3\}$ with $i < j$, let $u_{i,j} \in V(G)-\{v,u\}$ be a vertex that is adjacent to both $v_i$ and $v_j$. Such vertices exist since by Lemma \ref{inCommon}, $v_i$ and $v_j$ have at least three common neighbors other than $v$. Since $G$ has no $K_{3,3}$ subgraph by Lemma \ref{noK33} and $|N(u)\cap N(v)|=4$, the vertices $u_{1,2}$, $u_{1,3}$, and $u_{2,3}$ are distinct. Write $U \coloneqq \{u_{1,2}, u_{1,3}, u_{2,3}\}$, and $H \coloneqq G[N[v]\cup U\cup \{u\}]$. Then $d_H(v_4)\leq 2$ since $G$ has no $K_{3,3}$ subgraph. So since $\delta(G) \geq 4$ by Lemma \ref{minDegree}, there exists a component of $G-V(H)$ with neighbor $v_4$. Let $C$ be the vertex set of such a component. Observe that $N(C) \subseteq N(v) \cup U \cup \{u\}$.

Now we show that either $N(v) \subseteq N(C)$ or $U \cup \{u\} \subseteq N(C)$. Suppose otherwise. Then for all $i \in \{0,1\}$, we have $|X_i\cap N(C)| \leq 3$. Then by Lemma \ref{connectivity} applied to the separation $(C \cup N(C), V(G)-C)$, it follows that $|N(C)| = 6$ and $\Delta(G[N(C)]) \leq 1$. Then since $|N(C) \cap N(v)|=3$ and $|N(u) \cap N(v)|=4$, we have $u \notin N(C)$. Then $U \subseteq N(C)$. But $|N(C) \cap \{v_1, v_2, v_3\}| \geq 2$, which is a contradiction since $\Delta(G[N(C)])\leq 1$. We have shown that either $N(v) \subseteq N(C)$ or $U\cup \{u\} \subseteq N(C)$.

If $N(v) \subseteq N(C)$, let $G'$ be the graph formed from $G$ by contracting $C$ to a single vertex with neighborhood $N(v)$ and deleting all other vertices in $G-V(H)$. Let $G''$  be the graph formed from $G'$ by contracting $u_{1,2}$ to $v_1$, $u_{2,3}$ to $v_2$, and $u_{1,3}$ to $v_3$. Then $G''$ is isomorphic to $K_6^{\Delta Y}$, a contradiction.

So we may assume that $U \cup \{u\}\subseteq N(C)$. Remember also that by the choice of $C$, we have $v_4 \in N(C)$. Now let $G'$ be the graph formed by contracting $C$ to a vertex with neighborhood $U \cup \{u,v_4\}$ and deleting all other vertices in $G-V(H)$. Let $G''$ be the graph formed from $G'$ by contracting $v$ to $v_4$ and by contracting $u_{1,2}$ to $v_1$, $u_{2,3}$ to $v_2$, and $u_{1,3}$ to $v_3$. Then $G''$ is isomorphic to $K_6$, a contradiction.
\end{proof}

We are now ready to show that $G$ does not have additional kinds of separations of small order.
\begin{lemma}
\label{connectivity2}
Let $(A, B)$ be a non-trivial separation of $G$. If there exists $i \in \{0,1\}$ such that $|X_i \cap A \cap B| \leq 4$ and $|X_{1-i}\cap A \cap B| \leq 2$, then either $|A-B| = 1$ or $|B-A| = 1$.
\end{lemma}
\begin{proof}
Suppose otherwise. Let $(A,B)$ be a separation of minimum order that violates the lemma. Write $S \coloneqq A \cap B$ for convenience. Without loss of generality we assume that $|X_0 \cap S| \leq 4$ and $|X_{1}\cap S| \leq 2$. By Lemma \ref{connectivity}, we have $|X_0 \cap S|=4$.

First we will show that there exists a component of $G[A-B]$ with neighborhood $S$. Suppose otherwise. Let $C$ be the vertex set of any component of $G[A-B]$. If $|C|\geq 2$ and $N(C) \neq S$, then $(C \cup N(C), V(G)-C)$ is a separation violating the lemma of smaller order, a contradiction to the choice of $(A, B)$. So $|C|=1$ and thus since $\delta(G) \geq 4$ by Lemma \ref{minDegree}, we have $N(C) = X_0 \cap S$. Then since $|A-B| \geq 2$, the graph $G[A-B]$ has another component with vertex set $C'$ also consisting of a single vertex with neighborhood $X_0 \cap S$. But this is a contradiction to Lemma \ref{no44}. This shows that there exists a component of $G[A-B]$ with neighborhood $S$. By symmetry the same holds for $G[B-A]$. We now proceed by cases.

\begin{flushleft}
\textit{Case}: Either $|S|=4$, or $|S|=5$ and $|E(G[S])|=4$
\end{flushleft}
We will construct a super-separation $(G_A, G_B)$ that contradicts Lemma \ref{connectivity}. This will be the first time that $V(G_A)\not\subseteq A$.

Let $G_A$ be the graph formed from $G[A]$ by adding a single vertex, call it $a$, with neighborhood $X_0 \cap S$. We can see that $G_A$ is a minor of $G$ by contracting a component of $G[B-A]$ with neighborhood $S$ to a single vertex. Furthermore, $G_A$ is bipartite, has fewer vertices than $G$ since by assumption $|B-A|\geq 2$, and has at least five vertices since $A \subseteq V(G_A)$. Define $G_B$ analogously, by adding a single vertex with neighborhood $S \cap X_0$ to $G[B]$.

Suppose $G_A$ is isomorphic to $K_{3,t}$ for some $t$. Then since $N_{G_A}(a) = S \cap X_0$ and $|S\cap X_0|=4$, there exist two vertices $u$ and $v$ in $G_A-N_{G_A}[a]$ with degree exactly 4 in $G_A$. Then since $|S|\leq 5$, we have $|S-N_{G_A}[a]|\leq 1$, so at least one of the vertices, say $v$, is in $A-S$. Then $v$ has degree 4 in $G$, and $|N(u)\cap N(v)|=4$, a contradiction to Lemma \ref{no44}. By symmetry, we have shown that neither $G_A$ nor $G_B$ is isomorphic to $K_{3,t}$ for any $t$. Furthermore, the order of the super-separation $(G_A, G_B)$ is $|S|+2$, and 
\begin{align*}
|E(G_A)|+|E(G_B)| -|E(G)| &= |E(G[A])|+|E(G[B])|+8 -|E(G)| = |E(G[S])|+8.
\end{align*}
So by Lemma \ref{separation} applied to the super-separation $(G_A, G_B)$, we have $3(|S|+2)\geq |E(G[S])|+19$. But this is a contradiction since either $|S| \leq 4$, or $|S|=5$ and $|E(G[S])|=4$.

\begin{flushleft}
\textit{Case}: Either $|S|=5$ and $|E(G[S])|\leq 3$, or $|S|=6$
\end{flushleft}

This case is similar to the proof of Lemma \ref{connectivity}. Let $z \in X_1 \cap S$ so that $d_{\overline{G[S]}}(z)$ is maximum. Let $G_A$ be the graph formed from $G[A]$ by adding edges between $z$ and every vertex in $N_{\overline{G[S]}}(z)$. We can see that $G_A$ is a minor of $G$ by contracting a component of $G[B-A]$ with neighborhood $S$ to $z$. The graph $G_A$ is bipartite and has fewer vertices than $G$. Since there is a component of $G[A-B]$ with neighborhood $S$ and both $X_0 \cap S$ and $X_1 \cap S$ are non-empty, both $X_0 \cap (A-B)$ and $X_1 \cap (A-B)$ are also non-empty. Thus, since $\delta(G) \geq 4$, $A$ has at least four vertices in each side of the bipartition of $G$. So $G_A$ is not isomorphic to $K_{3,t}$ for any $t$. Define $G_B$ analogously, by adding edges between $z$ and every vertex in $N_{\overline{G[S]}}(z)$ to $G[B]$. By symmetry, neither $G_A$ nor $G_B$ is isomorphic to $K_{3,t}$ for any $t$.

The order of the super-separation $(G_A, G_B)$ is $|S|$, and
\begin{align*}
|E(G_A)|+|E(G_B)|-|E(G)|&=|E(G[A])|+|E(G[B])|+2d_{\overline{G[S]}}(z) -|E(G)|
\\&= |E(G[S])|+2d_{\overline{G[S]}}(z).
\end{align*}Then by Lemma \ref{separation}, $3|S|\geq |E(G[S])|+2d_{\overline{G[S]}}(z) + 11$. Observe that $$4|X_1\cap S| = \sum_{x \in X_1 \cap S}\left( d_{G[S]}(x)+ d_{\overline{G[S]}}(x)\right).$$So if $|S|=5$ and $|E(G[S])| \leq 3$, then $d_{\overline{G[S]}}(z)\geq 1$ and so $|E(G[S])|+2d_{\overline{G[S]}}(z)\geq 5$. This is a contradiction. So $|S|=6$. But then $|E(G[S])|+2d_{\overline{G[S]}}(z)\geq 8$, which is again a contradiction.
\end{proof}

By Lemmas \ref{minDegree} and \ref{no5}, the graph $G$ has a vertex of degree 4. Fix $v \in V(G)$ a vertex of degree 4, and write $N(v)=\{v_1, v_2, v_3, v_4\}$. Without loss of generality assume that $v \in X_0$. Choose a set $U \subseteq V(G)-N[v]$ of minimum cardinality such that either:
\begin{enumerate}

\item $U$ consists of a single vertex $u$ with $|N(u)\cap N(v)|=3$, or 
\item $U = \{u_{1,2}, u_{1,3}, u_{2,3}\}$ and for all $i,j \in \{1,2,3\}$ with $i<j$, $N(u_{i,j})\cap N(v) = \{v_i, v_j\}$.
\end{enumerate}
First we show that such a set exists. If there exists a vertex $u \in V(G)-N[v]$ such that $|N(u)\cap N(v)|\geq 3$, then by Lemma \ref{no44} in fact $|N(u)\cap N(v)|= 3$ and we are done. So we may assume that for all $u \in V(G)-N[v]$ we have $|N(u)\cap N(v)|\leq 2$. Then for all $i,j \in \{1,2,3\}$ with $i<j$, let $u_{i,j}$ be a vertex not in $N[v]$ that is adjacent to both $v_i$ and $v_j$. Such a vertex exists since $v_i$ and $v_j$ have at least three common neighbors other than $v$ by Lemma \ref{inCommon}. By assumption $u_{1,2}$, $u_{1,3}$, and $u_{2,3}$ are distinct and $N(u_{i,j})\cap N(v) = \{v_i,v_j\}$. So such a set exists. 

Write $H \coloneqq G[N[v]\cup U]$. Next we show one short lemma.

\begin{lemma}
\label{minor}
There do not exist disjoint sets $A, B \subseteq V(G)-V(H)$ such that $G[A]$ and $G[B]$ are connected, and $N(v) \subseteq N(A)$ and $N(v) \subseteq N(B)$.
\end{lemma}
\begin{proof}
Let $G'$ be the graph obtained from $G$ by contracting $A$ to a single vertex with neighborhood $N(v)$, contracting $B$ to a single vertex with neighborhood $N(v)$, and deleting all other vertices in $G-V(H)$. 

If $|U|=1$, then $G'$ is isomorphic to $K_{4,4}^-$, a contradiction. If $|U|=3$, then let $G''$ be the graph formed from $G'$ by contracting $u_{i,j}$ to $v_i$ for all $i,j \in \{1,2,3\}$ with $i<j$. Then $G''$ is isomorphic to $K_6^{\Delta Y}$, a contradiction.
\end{proof}

In the next lemma we show that $G-V(H)$ is connected and has a 1-separation satisfying certain properties.

\begin{lemma}
\label{connectedCutVertex}
The graph $G-V(H)$ is connected. Furthermore, there exist $\{a_0, a_0', a_1, a_1'\} \subseteq V(G)-V(H)$ and a 1-separation $(A_0, A_1)$ of $G-V(H)$ such that for every $i \in \{0,1\}$, we have $a_i, a_i' \in A_i$ and $a_i$ and $a_i'$ are both adjacent to $v_{2i+1}$ and $v_{2i+2}$.
\end{lemma}
\begin{proof}
First we will show that $G-V(H)$ is connected. Otherwise, by Lemma \ref{minor}, there exists a component of $G-V(H)$ with vertex set $C$ so that $N(v) \nsubseteq N(C)$. So by Lemma \ref{connectivity} applied to the separation $(C\cup N(C), V(G)-C)$, we find that $|N(C)|=6$ and $\Delta(G[N(C)])\leq 1$. It follows that $|U|=3$ and $U \subseteq N(C)$. This is a contradiction since $|N(C) \cap \{v_1, v_2, v_3\}| \geq 2$ and $\Delta(G[N(C)])\leq 1$. So $G-V(H)$ is connected.

Now for every $i \in \{0,1\}$, let $a_i$ and $a_i'$ be distinct vertices in $V(G)-V(H)$ that are adjacent to both $v_{2i+1}$ and $v_{2i+2}$. Such vertices exist since by Lemma \ref{inCommon} the vertices $v_{2i+1}$ and $v_{2i+2}$ share at least three common neighbors other than $v$, and by the definition of $U$ they share no more than one common neighbor in $U$. Furthermore by Lemma \ref{no44}, in fact $a_0, a_0', a_1, a_1'$ are all distinct. By Menger's Theorem, either the desired 1-separation exists, or $G-V(H)$ contains vertex-disjoint paths $P$ and $P'$ so that both $P$ and $P'$ have one end in $\{a_0, a_0'\}$ and one end in $\{a_1, a_1'\}$. But then by choosing $A \coloneqq V(P)$ and $B \coloneqq V(P')$ we have a contradiction to Lemma \ref{minor}.
\end{proof}

Fix $\{a_0, a_0', a_1, a_1'\} \subseteq V(G)-V(H)$ and a 1-separation $(A_0, A_1)$ of $G-V(H)$ as in the lemma. Let $a$ be the unique vertex in $A_0\cap A_1$, and for convenience write $H' \coloneqq G[V(H) \cup \{a\}]$. Let $C$ be the vertex set of a component of $G-V(H')$ so that $1 \leq |N(C) \cap N(v)|\leq 3$. Subject to this, choose $C$ such that $|N(C)|$ is minimum.

To see that such a component exists, for every $i \in \{0,1\}$, let $C_i$ be the vertex set of a component of $G[A_i-\{a\}]$ with $C_i \cap \{a_i, a_i'\} \neq \emptyset$. Then $G[C_0]$ and $G[C_1]$ are distinct components of $G-V(H')$. By Lemma \ref{minor}, either $N(v) \nsubseteq N(C_0)$ or $N(v) \nsubseteq N(C_1)$. So such a component exists. We first show:

\begin{lemma}
$|U| =3$
\end{lemma}
\begin{proof}
Suppose $|U|=1$. Let $u$ be the unique vertex in $U$. Without loss of generality we may assume that $N(u) \cap N(v) = \{v_1, v_2, v_3\}$. Remember that $v \in X_0$. If $a \in X_0$, then for every $i \in \{0,1\}$, we have $|N(C)\cap X_i| \leq 3$. But $|N(C)|\leq 5$, which is a contradiction to Lemma \ref{connectivity}. Thus $a \in X_1$. We prove the following claim:
\begin{claim}
Let $C'$ be the vertex set of a component of $G-V(H')$ so that $N(v) \nsubseteq N(C')$. Then $C'$ consists of a single vertex of degree 4 that is only adjacent to vertices in $N(v) \cup \{a\}$.
\end{claim}
\begin{proof}
Let $C'$ be the vertex set of such a component. Then $|N(C')\cap X_1|=|N(C')\cap(N(v)\cup \{a\})|\leq 4$ and $|N(C')\cap X_0| = |N(C') \cap U|\leq 1$. Note that $|V(G)-V(C')|\geq |N[v]-V(C')|\geq 2$. Then by Lemma \ref{connectivity2} applied to the separation $(C'\cup N(C'), V(G)-C')$, the set $C'$ consists of a single vertex. Then since $\delta(G)\geq 4$ by Lemma \ref{minDegree}, the vertex in $C'$ is only adjacent to vertices in $N(v)\cup\{a\}$.
\end{proof}

Now define the set
\begin{align*}W \coloneqq \{w \in V(G)-V(H'): d_G(w)=4 \textrm{ and } N(w)\subseteq N(v)\cup \{a\}\}.\end{align*}

Since $G-V(H)$ is connected by Lemma \ref{no44}, every vertex $w \in W$ is adjacent to $a$. Furthermore, every vertex $w \in W$ is adjacent to $v_4$, as otherwise $G[\{v, v_1, v_2, v_3, u, w\}]$ is isomorphic to $K_{3,3}$, a contradiction to Lemma \ref{noK33}.

Now we show that $|W| \geq 2$. By the claim and the choice of $C$, we have $|C|=1$. Since $a \in X_1$ while $\{a_0, a_0', a_1, a_1'\} \subseteq X_0$, for every $i \in \{0,1\}$ we have $|A_i-(\{a\} \cup C)|\geq |\{a_i, a_i'\}-C|\geq 1$. So $G-V(H')$ has at least three components. So by Lemma \ref{minor} and by the claim, $|W| \geq 2$.

Next we show that $G-(V(H')\cup W)$ has a component with vertex set $D$ so that $N(v) \cup \{a\} \subseteq N(D)$. By Lemma \ref{inCommon}, the vertices $v_1$ and $v_2$ have at least three common neighbors besides $v$. Suppose that there are distinct vertices $w, w' \in W$ that are common neighbors of $v_1$ and $v_2$. Then $G[\{v, w, w', v_1, v_2, v_4\}]$ is isomorphic to $K_{3,3}$, a contradiction to Lemma \ref{noK33}. So the vertices $v_1$ and $v_2$ have a common neighbor in $V(G)-(V(H')\cup W)$. So $V(G)-(V(H')\cup W)$ is non-empty. Let $D$ be the vertex set of any component of $G-(V(H')\cup W)$. Since $G-V(H)$ is connected, we have $a \in N(D)$. If $N(v) \nsubseteq N(D)$, then by the claim, $D$ consists of a single vertex of degree 4 that is only adjacent to vertices in $N(v)\cup \{a\}$. But this is a contradiction to the choice of $W$. So $N(v) \cup \{a\} \subseteq N(D)$. 

Let $w$ and $w'$ be distinct vertices in $W$. Since $G$ has no $K_{3,3}$ subgraph, we may assume without loss of generality that $N(w) = \{v_1, v_2, v_4, a\}$ and $N(w') = \{v_1, v_3, v_4, a\}$. Then let $G'$ be the graph formed from $G$ by contracting $D$ to a single component with neighborhood $N(v) \cup \{a\}$ and deleting all other vertices except $V(H') \cup \{w,w'\}$. Then let $G''$ be the graph formed from $G'$ by contracting $w$ to $v_2$, $u$ to $v_3$, $v$ to $v_4$, and $w'$ to $a$. Then $G''$ is isomorphic to $K_6$, a contradiction. This completes the proof of the lemma.
\end{proof}

So $|U|=3$. Write $U = \{u_{1,2}, u_{1,3}, u_{2,3}\}$ so that for all $i,j \in \{1,2,3\}$ with $i<j$, $N(u_{i,j})\cap N(v) = \{v_i, v_j\}$. By the choice of $U$, no vertex other than $v$ is adjacent to three or more vertices in $N(v)$. For convenience write $T \coloneqq N_{G[N(C)]}(a)\cup N_{\overline{G[N(C)]}}(a)$. That is, $T$ is the set of all vertices in $N(C)$ that are in the other side of the bipartition of $G$ from that of the vertex $a$. Let $x$ be some vertex in $N(v)-N(C)$. Such a vertex exists since by the choice of $C$ we have $|N(C) \cap N(v)|\leq 4$.

Now we give an overview of the rest of the proof. The goal is to show a contradiction to Lemma \ref{separation} on super-separations of $G$. Note that since $|N(C) \cap N(v)|\leq 3$, we have $N(C) \subsetneq N(v) \cup U \cup \{a\}$. So $|N(C)|\leq 7$. The previous lemmas on separations of $G$, Lemmas \ref{connectivity} and \ref{connectivity2}, apply only to separations of order six or less, so some casework is required to show a contradiction. We will frequently construct a super-separation $(G_C, G')$ so that $G[C \cup N(C)]$ is a subgraph of $G_C$ and $G-C$ is a subgraph of $G'$.

We first show a straightforward lemma that will help with constructing such super-separations. We are then able to show the harder lemma that $a \in X_1$ and that $G-(V(H')\cup C)$ is connected. Then it is easy to show that $v_4 \notin N(C)$, or else $G$ has a $K_6$ minor. A final lemma shows that certain vertices in $U$ have no neighbor in $V(G)-(V(H')\cup C)$. We then construct one last super-separation of $G$ that gives a contradiction to Lemma \ref{separation}, completing the proof. We begin with the following lemma.

\begin{lemma}
\label{aboutC}
The following hold:
\begin{enumerate}
\item The set $C$ has at least two vertices. Both $C \cup N(C)$ and $V(G)-C$ have at least four vertices in each side of the bipartition of~$G$.
\item Every neighbor of $v$ is adjacent to a vertex in $V(G)-(V(H')\cup C)$.
\item $|N(C) \cap N(v)|=3$ and $|T|=3$
\end{enumerate}
\end{lemma}
\begin{proof}
First we show that $|C|\geq 2$ and $|V(G)-(V(H')\cup C)|\geq 2$. We have $|V(G)-(C\cup N(C))|\geq |N[v]-(C\cup N(C))|\geq 2$. Suppose $|C|=1$. Since $|N(C)\cap N(v)|\geq 1$ by the choice of $C$, it follows that $N(C) \subseteq N(v) \cup \{a\}$. But then since $\delta(G)\geq 4$ by Lemma \ref{minDegree}, we have $|N(C) \cap N(v)|\geq 3$, a contradiction to the choice of $U$.

Next we show (i). The set $V(G)-C$ has at least four vertices in each side of the bipartition of $G$ since $V(H) \subseteq V(G)-C$. Since $|C|\geq 2$ and $G[C]$ is connected, the set $C$ is not contained in one side of the bipartition of $G$. Thus, since $\delta(G)\geq 4$, the set $C \cup N(C)$ has at least four vertices in each side of the bipartition of $G$.

Now we show (ii). Let $y$ be any vertex in $N(v)-\{x\}$. By Lemma \ref{inCommon}, the vertices $x$ and $y$ share at least three common neighbors other than $v$. By the choice of $U$, they share no more than two common neighbors in $U \cup \{a\}$. So $x$ and $y$ have a common neighbor in $V(G)-(V(H')\cup C)$.

Finally we show (iii). If $a \in X_0$ and $|N(C) \cap N(v)|\leq 2$, then this is a contradiction to Lemma \ref{connectivity2} applied to the separation $(C \cup N(C), V(G)-C)$. If $a \in X_1$ and $|N(C) \cap N(v)|\leq 2$, then by Lemma \ref{connectivity}, we have $|N(C)|=6$ and $\Delta(G[N(C)])\leq 1$. This is a contradiction since then $U \subseteq N(C)$ and $|N(C)\cap N(v)|=2$, but every vertex in $\{v_1, v_2, v_3\}$ has two neighbors in $U$.

If $a \in X_0$, then $T = N(C) \cap N(v)$ and so $|T|=3$ by the last paragraph. If $a \in X_1$ and $|T|\leq 2$, then this is a contradiction to Lemma \ref{connectivity2} applied to the separation $(C \cup N(C), V(G)-C)$.
\end{proof}

Next we show the following lemma.

\begin{lemma}
\label{C'}
$a \in X_1$ and $G-(V(H')\cup C)$ is connected.
\end{lemma}
\begin{proof}
Suppose otherwise. That is, suppose that $a \in X_0$ or $G-(V(H')\cup C)$ is not connected. Then:

\begin{claim}
\label{componentExists}
If $a \in X_1$, then there exists a component of $G-(V(H')\cup C)$ with vertex set $C'$ so that $U \subseteq N(C')$ and $|N(C') \cap N(v)|=3$.
\end{claim}
\begin{proof}
Suppose $a \in X_1$. Then $G-(V(H')\cup C)$ is not connected, and so by Lemma \ref{minor}, it has a component with vertex set $C'$ such that $N(v) \nsubseteq N(C')$. Since $C$ is the vertex set of a component of $G-V(H')$, it follows that $C'$ is also the vertex set of a component of $G-V(H')$. So $N(C')\cap X_0\subseteq U$ and thus $|N(C')\cap X_0|\leq 3$. So by Lemma \ref{connectivity} applied to the separation $(C'\cup N(C'), V(G)-C')$, we have that $|N(C')\cap X_1|\geq 2$. So $C'$ is the vertex set of a component of $G-V(H')$ such that $1\leq |N(C')\cap N(v)|\leq 3$. Then by the choice of $C$, we have $|N(C')| \geq |N(C)|$. By part (iii) of 
Lemma~\ref{aboutC}, 
since $a \in X_1$, we have $|N(C)|= 7$. Then $|N(C')|\geq 7$, and so $U \subseteq N(C')$ and $|N(C') \cap N(v)|=3$.
\end{proof}

Recall that by part (iii) of Lemma \ref{aboutC}, we have $|N(C)\cap N(v)|=3$. If $v_4 \notin N(C)$, then let $G_C$ denote the graph obtained from $G[C \cup N(C)]$ by adding edges between $a$ and every vertex in $N_{\overline{G[N(C)]}}(a)$. If $v_4 \in N(C)$, then let $G_C$ denote the graph obtained from $G[C \cup N(C)]$ by adding edges between $a$ and every vertex in $N_{\overline{G[N(C)]}}(a)$, and by adding edges between $v_4$ and every vertex in $N(x) \cap U \cap N(C)$.

Observe that in both cases, $G_C$ is bipartite and has fewer vertices than $G$. Furthermore, by part (i) of Lemma \ref{aboutC}, the graph $G_C$ is not isomorphic to $K_{3,t}$ for any $t$, and has at least five vertices. We now show two claims about the graph $G_C$.

\begin{claim}
\label{GCminor}
The graph $G_C$ is a minor of $G$.
\end{claim}
\begin{proof}
Recall that $G-V(H)$ is connected by Lemma \ref{connectedCutVertex}. So every component of $G-(V(H')\cup C)$ has $a$ as a neighbor.

First suppose that $a \in X_0$. By part (ii) of Lemma \ref{aboutC}, every vertex in $N(v)$ has a neighbor in $V(G)-(V(H')\cup C)$. Then we can see that $G_C$ is a minor of $G$ by contracting every component of $G-(V(H')\cup C)$ to $a$, and if $v_4 \in N(C)$, by contracting $x$ and $v$ to $v_4$. 

So we may assume that $a \in X_1$. Then by Claim \ref{componentExists}, there exists a component of $G-(V(H')\cup C)$ with vertex set $C'$ so that $U \subseteq N(C')$ and $|N(C') \cap N(v)|=3$. Then we can see that $G_C$ is a minor of $G$ by contracting $C'$ to $a$, and if $v_4 \in N(C)$, by contracting $x$ and $v$ to $v_4$. 
\end{proof}

\begin{claim}
$|E(G_C)|+|E(G-C)|-|E(G)|= 3+2|N(C)\cap U|$
\end{claim}
\begin{proof}
First suppose that $|N(C) \cap U|=3$. Observe that then if $v_4 \in N(C)$ we have $|N(x) \cap U\cap N(C)|=|N(x) \cap U|=2$. Let $\mathbbm{1}_{v_4 \in N(C)}$ be $1$ if $v_4 \in N(C)$ and $0$ otherwise. Then:
\begin{align*}
|E(G_C)|&+|E(G-C)|-|E(G)| \\
&= d_{\overline{G[N(C)]}}(a)+2\mathbbm{1}_{v_4 \in N(C)}+|E(G[C \cup N(C)])|+|E(G-C)|-|E(G)|\\
&=d_{\overline{G[N(C)]}}(a) +2\mathbbm{1}_{v_4 \in N(C)}+|E(G[N(C)])|\\
&= 2\mathbbm{1}_{v_4 \in N(C)}+|T|+|E(G[N(C)-\{a\}])|.
\end{align*}
In either case, since $|T|=3$ by part (iii) of Lemma \ref{aboutC}, $$2\mathbbm{1}_{v_4 \in N(C)}+|T|+|E(G[N(C)-\{a\}])|=9=3+2|N(C) \cap U|,$$
which completes the case that $|N(C) \cap U|=3$. 

So we may assume that $|N(C) \cap U|\leq 2$. Since $|T| \geq 3$, it follows that $a \in X_0$. Then by Lemma \ref{connectivity} applied to the separation $(C \cup N(C), V(G)-C)$, we have $|N(C) \cap U|=2$ and $\Delta(G[N(C)])\leq 1$. By symmetry between pairs of vertices in $U$, we may assume that $N(C) \cap U = \{u_{1,2},u_{1,3}\}$. Then $v_1 \notin N(C)$. Then $|N(x) \cap U \cap N(C)|=|N(v_1)\cap U \cap N(C)|=2$ and similarly to the last case we find that
\begin{align*}
|E(G_C)|+|E(G-C)|-|E(G)|&= |T|+2+|E(G[N(C)-\{a\}])|\\&=|T|+4 =3+2|N(C) \cap U|,
\end{align*}
which completes the proof of the claim.
\end{proof}

We now show that $G[N(C)]$ is almost a complete bipartite graph.

\begin{claim}
\label{claim:maxDegree}
$\Delta(\overline{G[N(C)]}) \leq 1$
\end{claim}
\begin{proof}
Suppose that $\Delta(\overline{G[N(C)]}) \geq 2$. Let $z \in N(C)$ be a vertex with maximum degree in $\overline{G[N(C)]}$. Then let $G'$ be the graph obtained from $G-C$ by adding an edge between $z$ and every vertex in $N_{\overline{G[N(C)]}}(z)$. We can see that $G'$ is a minor of $G$ on strictly fewer vertices by contracting $C$ to $z$. Furthermore, $G'$ is bipartite, and by part (i) of Lemma \ref{aboutC}, has at least five vertices and is not isomorphic to $K_{3,t}$ for any $t$. By Claim \ref{GCminor}, the graph $G_C$ is a minor of $G$. So since $(C \cup N(C), V(G)-C)$ is a separation of $G$, it follows that $(G_C, G')$ is a super-separation of $G$. 

In fact we have shown that $(G_C, G')$ is a non-trivial bipartite super-separation of $G$ so that both $G_C$ and $G'$ have at least five vertices and are not isomorphic to $K_{3,t}$ for any $t$. Since $|N(C) \cap N(v)|=3$ and $a \in N(C)$, the order of the super-separation $(G_C, G')$ is $4+|N(C) \cap U|$. Then by Lemma \ref{separation} and the last claim,
\begin{align*}
3(4+|N(C) \cap U|)&\geq |E(G_C)|+|E(G')|-|E(G)|+11\\
&\geq |E(G_C)|+|E(G-C)|-|E(G)|+13 \\
&= 2|N(C) \cap U|+16.
\end{align*}
But then $|N(C) \cap U|\geq 4$, which is a contradiction since $|U|=3$.
\end{proof}

We are now ready to complete the proof of the lemma. We proceed by cases.

\begin{flushleft}
\textit{Case}: $a \in X_0$
\end{flushleft}
By the last claim, $d_{\overline{G[N(C)]}}(a) \leq 1$. Then since $|N(C) \cap N(v)|=3$, $a$ is adjacent to at least two vertices in $N(v)$. So by the choice of $U$, $a$ is adjacent to exactly two vertices in $N(v)$. So there exists a vertex $y \in N(C)\cap N(v)$ that is not adjacent to $a$. By the last claim, $d_{\overline{G[N(C)]}}(y) \leq 1$. So since $y$ is not adjacent to $a$, $y$ is adjacent to every vertex in $N(C) \cap U$.

Then $U \nsubseteq N(C)$. Then by Lemma \ref{connectivity} applied to the separation $(C \cup N(C), V(G)-C)$, we have $|N(C) \cap U|=2$ and $\Delta(G[N(C)])\leq 1$. But then $1 \geq d_{G[N(C)]}(y)=3-d_{\overline{G[N(C)]}}(y) \geq 2$, a contradiction.

\begin{flushleft}
\textit{Case}: $a \in X_1$
\end{flushleft}
Then since $|T|=3$, we have $U \subseteq N(C)$. Suppose there exists $u \in U$ such that $ua \notin E(G)$. Then since $|N(C) \cap N(v)|=3$ and $u$ is adjacent to exactly two vertices in $N(v)$, it follows that $d_{\overline{G[N(C)]}}(u) \geq 2$, a contradiction to the last claim. So $U \subseteq N(a)$. Also by the last claim applied to $v_4$, we have $v_4 \notin N(C)$. So $N(C) = U \cup \{v_1, v_2, v_3, a\}$. 

By Claim \ref{componentExists}, there exists a component of $G-(V(H')\cup C)$ with vertex set $C'$ so that $U \subseteq N(C')$ and $|N(C')\cap N(v)|=3$. By symmetry between the vertices $v_1$, $v_2$, and $v_3$, we may assume that $v_1 \in N(C')$. Then let $G_C'$ be the graph formed from $G[C \cup N(C)]$ by adding an edge between $v_1$ and $u_{2,3}$. We can see that $G_C'$ is a minor of $G$ with strictly fewer vertices by contracting $C'$ to $v_1$. By part (i) of Lemma \ref{aboutC}, the graph $G_C'$ has at least five vertices and is not isomorphic to $K_{3,t}$ for any $t$.

Let $G'$ be the graph formed from $G-C$ by adding an edge between $v_1$ and $u_{2,3}$. We can see that $G'$ is a minor of $G$ on strictly fewer vertices by contacting $C$ to $v_1$. By part (i) of Lemma \ref{aboutC}, the graph $G'$ has at least five vertices and is not isomorphic to $K_{3,t}$ for any $t$.

Then $(G_C', G')$ is a non-trivial bipartite super-separation of $G$ of order $7$ such that neither $G_C'$ nor $G'$ is isomorphic to $K_{3,t}$ for any $t$. So by Lemma \ref{separation} and since $U \subseteq N(a)$ and $N(C) = U \cup \{v_1, v_2, v_3, a\}$,\begin{align*}
3(7)&\geq |E(G_C')|+|E(G')|-|E(G)| +11\\
&= |E(G[C \cup N(C)])|+|E(G-C)|-|E(G)|+13\\
&=|E(G[N(C)])|+13\\
&=22,
\end{align*}a contradiction. This completes the proof of the lemma.
\end{proof}

We have shown that $a \in X_1$ and that $G-(V(H')\cup C)$ is connected. For convenience write $D \coloneqq V(G)-(V(H')\cup C)$. By part (ii) of Lemma \ref{aboutC}, we have $N(v) \subseteq N(D)$. Also since $|T|=3$, we have $U \subseteq N(C)$. The final two lemmas show that certain vertices are not neighbors of $C$ or not neighbors of $D$.

\begin{lemma}
\label{v4}
$v_4 \notin N(C)$
\end{lemma}
\begin{proof}
Suppose $v_4 \in N(C)$. Remember that $U \cup \{a\} \subseteq N(C)$. Then let $G'$ be the graph formed from $G$ by contracting $D$ to a single vertex with neighborhood $N(v) \cup \{a\}$ and by contracting $C$ to a single vertex, call it $c$, with neighborhood $U \cup \{v_4, a\}$. Then let $G''$ be the graph formed from $G'$ by contracting $u_{1,2}$ to $v_1$, $u_{2,3}$ to $v_2$, $u_{1,3}$ to $v_3$, $v$ to $v_4$, and finally $c$ to $a$. Then~$G''$ is isomorphic to $K_6$, a contradiction.
\end{proof}

\begin{lemma}
\label{uNeighborD}
If $u \in U\cap N(D)$, then $ua \in E(G)$.
\end{lemma}
\begin{proof}
Suppose otherwise. Let $G_C$ be the graph formed from $G[C \cup N(C)]$ by adding edges between $u$ and all vertices in $N_{\overline{G[N(C)]}}(u)$. We can see that $G_C$ is a minor of $G$ by contracting $D$ to $u$, since $N(v) \cup \{a\} \subseteq N(D)$. Let $G'$ be the graph formed from $G-C$ by adding edges between $a$ and all vertices in $N_{\overline{G[N(C)]}}(a)$. We can see that $G'$ is a minor of $G$ by contracting $C$ to $a$. Then $(G_C, G')$ is a non-trivial bipartite super-separation of $G$ of order $|N(C)|=7$. By part (i) of Lemma \ref{aboutC}, both $G_C$ and $G'$ have at least four vertices on each side of the bipartition of $G$. So neither is isomorphic to $K_{3,t}$ for any $t$. So by Lemma \ref{separation} and since $u$ is adjacent to exactly two vertices in $N(v)$,
\begin{align*}
3(7)& \geq |E(G_C)|+|E(G')|-|E(G)| +11\\
&= d_{\overline{G[N(C)]}}(u)+d_{\overline{G[N(C)]}}(a)+|E(G[N(C)])| +11\\ 
&\geq |N(C)\cap U|+|E(G[N(C)-\{a\}])|+13.
\end{align*}
Since $v_4 \notin N(C)$ by Lemma \ref{v4}, we have $|E(G[N(C)-\{a\}])|=6$. This is a contradiction since $|N(C)\cap U)|=3$.
\end{proof}

Write $U' \coloneqq U \cap N(D)$. Let $G_C$ be the graph formed from $G[C \cup N(C)]$ by adding a vertex with neighborhood $\{v_1, v_2, v_3, a\}$. We can see that $G_C$ is a minor of $G$ on strictly fewer vertices by contracting $D$ to a single vertex and since $|D|\geq 2$ by the choice of $U$. Also, by part (i) of Lemma \ref{aboutC}, the graph $G_C$ has at least four vertices in each side of the bipartition of $G$.

Let $G'$ be the graph formed from $G[D \cup N(D) \cup \{v\}]$ by adding a vertex with neighborhood $\{v_1, v_2, v_3, a\}$. We can see that $G'$ is a minor of $G$ on strictly fewer vertices by contracting $C$ to a single vertex and since $|C|\geq 2$ by part (i) of Lemma \ref{aboutC}. Furthermore, $G'$ is not isomorphic to $K_{3,t}$ for any $t$ since $va \notin E(G')$.

Now we show that every edge of $H'$ is an edge of either $G_C$ or $G'$. Let $e$ be an edge of $H'$. If $e$ is incident to $v$, then $e$ is an edge of $G'$. If $e$ is incident to a vertex in $U$, then $e$ is an edge of the graph $G_C$. Furthermore, if $e$ is incident to a vertex in $U'$, then $e$ is also an edge of $G'$.

So $(G_C, G')$ is a non-trivial bipartite super-separation of $G$. Furthermore, neither $G_C$ nor $G'$ is isomorphic to $K_{3,t}$ for any $t$, and the order of the super-separation $(G_C, G')$ is $6+|U'|$. Remember also that by Lemma \ref{uNeighborD}, every vertex $u \in U'$ is adjacent to $a$. Then by Lemma \ref{separation},
\begin{align*}
3(6+|U'|) &\geq |E(G_C)|+|E(G')|-|E(G)|+11
=3|U'|+19,
\end{align*}
a contradiction. This completes the proof of Theorem~\ref{thm:main2}.
%\end{proof}

\xx{%
\section*{Acknowledgment}
We would like to thank two anonymous referees for carefully reading the manuscript and for providing
many helpful suggestions.}

\let\OLDthebibliography\thebibliography
\renewcommand\thebibliography[1]{
  \OLDthebibliography{#1}
  \setlength{\parskip}{0em}
  \setlength{\itemsep}{0pt plus 0.3ex}
}

\baselineskip 11pt
\vfill
\noindent
This material is based upon work supported by the National Science Foundation.
Any opinions, findings, and conclusions or
recommendations expressed in this material are those of the authors and do
not necessarily reflect the views of the National Science Foundation.
\end{document}